\documentclass{amsart}

\usepackage{amsfonts}
\usepackage{amssymb}
\usepackage{amscd}

\newtheorem{thm}{Theorem}[section]
\newtheorem{lem}[thm]{Lemma}

\newtheorem{cor}[thm]{Corollary}

\theoremstyle{definition}

\newtheorem{problem}{Problem}

\begin{document}

\title[C*-submodule preserving module mappings]{C*-submodule preserving module mappings on Hilbert C*-modules}
\author{Michael Frank}
\address{Hochschule f\"ur Technik, Wirtschaft und Kultur (HTWK) Leipzig, Fakult\"at Informatik und Medien, PF 301166, D-04251 Leipzig, Germany.}
\email{michael.frank@htwk-leipzig.de, michael.frank.leipzig@gmx.de}

\subjclass{Primary 46L08; Secondary 46C05, 47B49. }

\keywords{Hilbert C*-modules; C*-submodule preserving module mappings; mapping preserved subbimodule sets}

\dedicatory{in the memory of Gerard John Murphy}

\begin{abstract}
Let $A$ be a (non-unital, in general) C*-algebra with center $Z(M(A))$ of its multiplier algebra, and let $\{ X, \langle .,. \rangle \}$ be a full Hilbert $A$-module. Then any bijective bounded module morphism $T$, for which every norm-closed $A$-submodule of $X$ is invariant, is of the form $T=d \cdot {\rm id}_X$ where $d \in Z(M(A))$ is invertible. As an example of a merely injective bounded module operator with that preserver property serves $T =d \cdot {\rm id}_X$ where $|d| \in Z(M(A))$ has a positive spectrum, but not bounded away from zero. The same assertions are true if the restriction on the C*-submodules to be norm-closed is dropped. \newline
From a different point of view, for two given strongly Morita equivalent C*-algebras $A$ and $B$ and a Hilbert $B$-$A$ bimodule $\{ X, \langle .,. \rangle \}$ with faithful compact right action of $B$, for any two two-sided norm-closed ideals $I \subseteq A$, $J \subseteq B$,  any full compatible norm-closed Hilbert $J$-$I$ subbimodule of $X$ is invariant for any left bounded $B$-module operator and any right bounded $A$-module operator. So these subsets of submodules of $X$ cannot rule out any bounded module operator as a non-preserver of that subset collection, however any single element of this subset collection is preserved by any bounded module operator on $X$. \newline
For any $B$-$A$ imprimitivity bimodule both the C*-valued inner product values are always preserved by bijective bounded module operators $T$ on $X$ iff $T= u \cdot {\rm id}_X$ for a unitary element $u\in Z(M(A))$. 
\end{abstract}

\maketitle

One can easily show that injective bounded linear operators which preserve all Hilbert subspaces of a Hilbert space $H$ have to be equal to a certain non-zero complex multiple of the identity operator on $H$. They are surjective automatically. The relevant test spaces are the one-dimensional subspaces that are automatically norm-complete, and thereafter any two-dimensional subspace spanned by two non-colinear elements and, therein, the sum of them.  For injective bounded module operators preserving all norm-closed C*-submodules of a given Hilbert C*-module over a given C*-algebra $A$ the answer is much more difficult to find, cf.~Kamran Sharifi's problem \cite[Problem 2.16]{Sharifi}. 

\begin{problem} (\cite[Problem 2.16]{Sharifi})
Let $X$ be a (left) Hilbert $A$-$A$-bimodule. For which C*-algebras $A$ does the following assertion hold?

If every closed submodule of $X$ is invariant under a bounded adjointable operator $T: X \to X$ then $T$ is a $*$-coefficient morphism, i.e. there exists a non-zero element $a \in A$ such that $T= {\rm id}_X \cdot a$, where ${\rm id}_X$ is the identity operator on $X$ (and $a$ acts as right bounded module operator ${\rm id}_X \cdot a$ on $X$).
\end{problem} 

The formulation of the problem needs some explanations and more concrete settings: On one side, the set of submodules of $X$, each element of which should be $T$-invariant, has to be precisely described for any particular investigation. On the other side, the operator $T$ has to be characterized by specific mapping properties and by the part of properties of the submodules which has to be preserved. Moreover, the notion of a 'Hilbert $A$-$A$ bimodule' has to be defined more precise since several definitions of similar notions can be found in the literature. Last but not least, the switch to the C*-module point of view brings into consideration many more additional new situations, requiring separate treatments. For a seminal survey on preserver problems on algebraic structures of operators on linear spaces and on function spaces see L.~Moln\'ar's book \cite{Molnar_2007} and a survey paper by I.~Chalendar and J.~R.~Partington \cite{CP}. On orthogonality and isometry in Hilbert C*-modules see publications \cite{FMP_2011,LNW,Lance_1994,F_1997,Solel_2001}. A comprehensive survey on preserver problems in the research field of C*-algebras can be found in \cite{KST_2025}. More general approaches are discussed in \cite{LP_2001,LT_2024}.

We are going to give a comprehensive answer on the structure of such bounded module operators on suitable sets of $A$-submodules resorting to bijective operators in section 1, and to discuss Sharifi's assertion in several (bi)module situations in more detail in section 2. A complete solution for merely injective bounded module preservers is out of reach at present. The same situation appears for general $B$-$A$ correspondences for admissible pairs of C*-algebras $A$ and $B$. 

\section{Injective bounded module operators as C*-submodule preservers}

Let  $A$  be a C*-algebra and  $X$  be a complex-linear space and (left)  $A$-module where both the complex-linear structures are compatible. Let  $X$  admit a map  $\langle .,. \rangle: X \times X \to A$,  $A$-linear in the first and  conjugate-$A$-linear in the second variable such that
\begin{enumerate}
   \item $\langle x,y \rangle = \langle y,x \rangle^*$    for any  $x,y \in X$,
   \item $0 \leq \langle x,x\rangle$   for any  $x \in X$,
   \item $\langle x,x \rangle = 0$   iff   $x=0$.
\end{enumerate}
Then  $\|x\| := \| \langle x,x \rangle \|_A^{1/2}$  defines an $A$-module norm on  $X$. We define Hilbert $A$-modules to be full, if the norm-closure of the linear hull of the $A$-valued inner product values $\langle X,X \rangle$ coincides with $A$. 
For two suitable C*-algebras $A$ and $B$ we say that $A$ and $B$ are strongly Morita equivalent iff there exists a full left Hilbert $A$-module which is at the same time a full right Hilbert $B$-module, and the two C*-valued inner products on it fulfil the equality $\langle x,y \rangle_A z =  x \langle y,z \rangle_B$ for any $x,y,z \in X$.

We make use of a fact obtained by E.~C.~Lance to prove the main theorem:

\begin{lem}  {\rm (}\cite[Lemma]{Lance_1994}{\rm )} \label{Lance}
    Suppose, $a,b$ are positive elements of a C*-algebra, and that $\|ac\| = \|bc\|$  for any element $c$ of that C*-algebra. Then $a=b$.
\end{lem}

\begin{thm}  \label{one-sided}
    Let $A$ be a (non-unital, in general) C*-algebra with center $Z(M(A))$ of its multiplier algebra, and let $X$ be a full Hilbert $A$-module. Then any bijective bounded module morphism $T:X \to X$, for which every norm-closed $A$-submodule of $X$ is invariant, is of the form $T=d \cdot {\rm id}_X$ where $d \in Z(M(A))$ and invertible.
\end{thm}

\begin{proof}
Consider the $A$-submodules generated by a single element $x \in X$, i.e.~sets of the form $\{ ax : a \in A \}$. They might be not norm-closed, so we take their norm-closure $Y_x$ to continue. Now, consider two non-zero elements $x,y \in X$ that are orthogonal to each other. Obviously, the sets $\{ ax : a \in A \}$ and $\{ by : b \in A \}$ are orthogonal to each other too, and so are their norm-completions $Y_x$ and $Y_y$ by norm-continuity of the orthogonality relations. Since $T$ preserves $Y_x$ and $Y_y$ by supposition, the non-zero, by assumption, elements $T(x)$ and $T(y)$ are orthogonal to each other, too. Since $x,y \in X$ were arbitrarily chosen the bounded module map $T$ has to be orthogonality preserving on $X$. 

Since $X$ is supposed to be full there exists a (unique) positive invertible element $c \in Z(M(A))$ such that $\langle T(x),T(y) \rangle = c \langle x,y \rangle$ for any $x,y \in X$, see \cite[Thm.~3.2]{LNW} and \cite[Thm.~4]{FMP_2011}. Equivalently, there exist a (unique) positive element $c^{1/2} \in Z(M(A))$ and an isometric bijective module isomorphism $S$ of $X$ such that $T(x) = S(c^{1/2}x)$ for any  $x \in X$, cf.~\cite[Thm.~3.4(i)]{LNW}. The element $c^{1/2} \in Z(M(A))$ has to be invertible since $T$ is assumed to be bijective. Note that isometric module isomorphisms of Hilbert C*-modules are unitary module isomorphisms, cf.~\cite[Thm.]{Lance_1994}, \cite[Thm. 5]{F_1997}, \cite[Thm. 1.1]{Solel_2001}. Furthermore, the center of a multiplier algebra of a C*-algebra might be larger than the multiplier C*-algebra of the center of the initiating C*-algebra, e.g.~for non-unital C*-algebras and their unital multiplier C*-algebras. 

Every element $x \in X$ gives rise to a positive element $|x|=\langle x,x \rangle^{1/2} \in A$. Moreover, 
\begin{eqnarray}  \label{unitary}
     \| a \cdot |S(x)| \|  & = &   \| a \langle S(x),S(x) \rangle a^* \|^{1/2} = \| \langle S(ax),S(ax) \rangle \|^{1/2}   \\\nonumber
                          & = &   \| S(ax) \| = \| ax \|  \quad ({\rm isometry})\\\nonumber
                          & = &   \| a \langle x,x \rangle a^* \|^{1/2}  = \| a \cdot |x| \|
\end{eqnarray}           
for any $a \in A$, any $x \in X$. By Lemma  \ref{Lance}  we obtain $|S(x)|=|x|$  for any $x \in X$. We extend $X$ canonically to a larger Hilbert $A^{**}$-module using W.~L.~Paschke's modular tensor product construction in \cite[Section 4]{Paschke}. Note that $X$ is isometrically embedded. For Hilbert $A^{**}$-modules the $A^{**}$-valued inner product on them can always be extended to their $A^{**}$-dual Banach $A^{**}$-modules turning the latter into selfdual Hilbert $A^{**}$-modules with W*-algebras as their C*-algebras of all bounded (adjointable) $A^{**}$-linear operators on them, cf.~\cite[Thms.~3.2, 3.10, Prop.~3.6]{Paschke}. Bounded modular operators on $X$ admit a unique isometric extension to bounded modular operators on the derived selfdual Hilbert $A^{**}$-module, cf.~\cite[Section 4]{Paschke}. Beside this, surjective isometries of Hilbert C*-modules lift to the respective linking C*-algebras, cf.~\cite[Thm.~3.2]{Solel_2001}.

Consequently, the multiplier algebra of the linking algebra of that selfdual Hilbert $A^{**}$-module is a W*-algebra, and every element in it admits polar decomposition. So $u_x \cdot S(x)=x \in A$ for some unitary $u_x$ in the bidual von Neumann algebra $A^{**}$ of $A$ and for any $x \in X$, with reference to the fact that W.~L.~Paschke's construction is adapted to the universal $*$-representation of $A$ in $A^{**}$, cf.~\cite[Thm.~3.7.7, Prop.~3.7.8]{Ped2018}. For two different elements $x,y \in X$ one has
\begin{eqnarray*}
     x+y &=& u_{x+y} S(x+y) = u_{x+y} (S(x)+S(y)) \\
           &=& u_x S(x) + u_y S(y)
\end{eqnarray*}
Since $X$ is assumed to be full and $x,y \in X$ are arbitrarily chosen the unitary elements $u_x$ do not depend on $x \in X$, but only on the surjective isometry $S$. Consequently, since $\{ ax: a \in A, x \in X \}$ is invariant w.r.t.~the action of $S$, norm-closed or not, we have two equalities:
\begin{eqnarray*}
    u \cdot S(ax) & = & u \cdot a \cdot S(x) \quad (S \,\, {\rm is} \,\, {\rm  {C^*-}linear})  \\
                       & = & a\cdot x  = a \cdot u \cdot S(x) \quad ({\rm see} \,\, {\rm (\ref{unitary}))}
\end{eqnarray*}
Therefore, $u \cdot a = a \cdot u$ for any $a \in A$ in $A^{**}$, since $S$ is a surjective isometry, $x \in X$ is arbitrary and $X$ is assumed to be full. So, $u$ belongs to the center of $A^{**}$. By (\ref{unitary}) we have $u \cdot a \cdot S(x) = a \cdot u \cdot S(x) = ax \in X$ for any $a \in A$, $x \in X$. We know $S$ is bijective, $A$ admits its universal $*$-representation in $A^{**}$ and $X$ is full. In this situation the element $u$ has to be a multiplier of $A$, cf.~\cite[Theorem]{Ped}, \cite[Prop.~3.12.3]{Ped2018}. The intersection of the C*-algebra of multipliers of $A$ with the center of $A^{**}$, calculated in $A^{**}$, is the center $Z(M(A))$ of the multiplier algebra of $A$. Hence, $x=u \cdot S(x)$ in $X$, where $u \in Z(M(A))$ unitary. The factors $c^{1/2}$ and $u$ of $Z(M(A))$ can be multiplied and singled out. We arrive at $T= d  \cdot {\rm id}_X$ with $d=c^{1/2}u \in Z(M(A))$. Since $T$ is bijective the factor $d$ has to be invertible.
\end{proof}

While for Hilbert spaces finite-dimensional linear subspaces are norm-closed automatically, this is not the case for most pairings of a non-unital C*-algebra $A$, a full Hilbert $A$-module $X$ and certain singly generated $A$-submodule of $X$. So, we get a similar preserver theorem requesting the invariance of all $A$-submodules, norm-closed or not. 

\begin{cor}  \label{one-sided-all}
    Let $A$ be a (non-unital, in general) C*-algebra with center $Z(M(A))$ of its multiplier algebra, and let $X$ be a full Hilbert $A$-module. Then any bijective bounded module morphism $T: X \to X$, for which every $A$-submodule of $X$ is invariant, is of the form $T=d \cdot {\rm id}_X$ where $d \in Z(M(A))$ is invertible.
\end{cor}

As a difference to  the Hilbert space situation, injective bounded module operators on Hilbert C*-modules which preserve any (closed) Hilbert C*-submodule of them need not to be surjective, in particular, if $Z(M(A)))$ is infinite-dimensional. Indeed, suppose $T = d \cdot {\rm id}_X$ where $|d| \in Z(M(A))$ admits a strictly positive spectrum which is not bounded away from zero. Then $T$ is injective and $T$ is not surjective but still preserves all Hilbert $A$-submodules. This cannot happen if $Z(M(A)))$ is finite-dimensional. 

Non-surjective (C*-)linear isometries of C*-algebras or of Hilbert C*-modules have still to be investigated. In any case, they are injective. For such maps on C*-algebras first structural results were obtained by Cho-Ho Chu and Ngai Ching Wong in \cite{Chu_Wong}.

\section{Invariant C*-submodules for bounded module operators}

We are going to investigate a slight generalization of \cite[Problem 2.16]{Sharifi} adressing its bimodule aspects. In the sequel, we consider Hilbert C*-modules as two-sided Hilbert C*-modules. To fix some notions according to the contemporary use of them (cf.~\cite{aHNS}), for any two given C*-algebras $B$ and $A$, a left-Hilbert $B$-$A$ bimodule $X$ is a left Hilbert $A$-module together with a right action of $B$ by adjointable operators on $X$ that is implemented by a $*$-homomorphism $\phi: B \to {\rm End}_A(X)$ (i.e. $x \cdot b := \phi(b)(x)$). If the $*$-homomorphism $\phi$ implementing the right action of $B$ takes values in the C*-algebra ${\rm K}_A(X)$ of ``compact'' operators on $X$, we say that the Hilbert $B$-$A$ bimodule $X$ has compact right action. In case $\phi$ is a $*$-isomorphism of $B$ and of ${\rm K}_A(X)$ we call it a faithful compact right action. For a non-trivial norm-closed two-sided ideal $J \in B$ the compact right action of $J$ on $X$ is compatible if $\phi(B) \subseteq {\rm K}_A(X)$ and $J$ acts as $\phi(J)$ on $X$. 
In the even more particular case, that a $B$-$A$ Hilbert bimodule $X$ admits a faithful compact right action of $B$ on $X$ and the canonical two $A$-valued and $B$-valued inner products fulfil the equality  $\langle x,y \rangle_A z = x \langle y,z \rangle_B$ for any $x,y,z \in X$ we call $X$ a $B$-$A$ imprimitivity bimodule (identifing $B$ with $\phi(B) = {\rm K}_A(X)$), cf.~\cite{RW_1998}. In the latter situation $A$ and $B$ are strongly Morita equivalent, so not any pair $A$, $B$ of C*-algebras admits connecting $B$-$A$ imprimitivity bimodules $X$. Also, for a C*-algebra $A$ the set of all $A$-$A$ imprimitivity bimodules forms a (non-trivial, in general) group with the inner tensor product as operation, the strong Picard group, cf.~\cite{Landsman_2001, BW}. The unit of this group is $A$ considered as a Hilbert $A$-module over itself with $A$-valued inner product induced by multiplication and $*$-operation, an often used example. As a consequence, the set of all pairwise unitarily non-isomorphic $B$-$A$ imprimitivity bimodules can be quite manifold, cf.~\cite{Rae_1981}. 

Let us consider the set of all full Hilbert $J$-$I$ submodules $Y$ of $X$ with faithful compact right actions of $J$ for all suitable pairs of norm-closed two-sided essential ideals $I \subseteq A$ and $J\subseteq B$. The bijective correspondences of the partially ordered lattices of norm-closed two-sided (essential) ideals $I \in A$, resp.~$J \in B$, as well as of the partially ordered lattice of Hilbert $J$-$I$ subbimodules $Y$ of $X$ that are $J$-$I$ imprimitivity bimodules, are described in \cite[sections 2 and 3]{RW_1998} and in \cite[Thm.~2.1]{KQW_2024}. 

We must recall some facts about multiplier algebras of norm-closed two-sided essential ideals in C*-algebras. So for two norm-closed two-sided essential ideals $I_1\subseteq I_2 \subseteq A$ we have always injective $*$-isomorphisms of the respective multiplier algebras leading to isometric $*$-homomorphisms $M(A) \subseteq M(I_2) \subseteq M(I_1)$, cf.~\cite[Prop.~1.2.20]{AM_2003}. Recall that the set of all bounded $A$-linear operators on $X$ is isometrically-algebraically isomorphic to the right multiplier algebra of $B$, $RM(B)={\rm End}_A(X)$, and symmetrically the set of all bounded $B$-linear operators on $X$ is isometrically-algebraically isomorphic to the left multiplier algebra of $A$, $LM(A)={\rm End}_B(X)$, cf.~\cite[Thm.~1.5]{Lin_1991}. One has to be aware that there exist full Banach C*-modules $X$ that admit at least two $A$-valued inner products which induce equivalent norms, but which are not unitarily isomorphic, \cite{CC}, \cite[Ex.~2.3]{Lin_1992}, \cite[Ex.~6.2]{Brown}, \cite[Ex.~3.1, Ex.~4.3, Ex.~7.2]{F_1999}. The respective C*-algebras of ``compact'' $A$-linear operators and their multiplier algebras of all adjointable bounded $A$-linear operators are pairwise not $*$-isomorphic,  \cite[Prop.~5.3(iii)]{F_1999}. However, the derived two operator norms induced on the respective common Banach algebra of all bounded $A$-linear operators are equivalent, and by Huaxin Lin's theorem, the respective right multiplier algebras are always isomorphic as Banach algebras. As a consequence, we have to fix an $A$-valued inner product, and we have to consider Hilbert $A$-(sub)modules always as a pair of the Hilbert $A$-module itself and of the $A$-valued inner product under consideration. Since the two induced respective C*-algebras of all ``compact'' $A$-linear operators are strongly Morita equivalent (via $A$) they admit isomorphic lattices of norm-closed two-sided essential ideals, so there is no loss of generality of the subsequent considerations here.

The following statement is obvious for Hilbert C*-modules $X$ over simple C*-algebras $A$, in particular, for Hilbert spaces, because their only norm-closed ${\rm K}_A(X)$-$A$ subbimodules are $\{ 0 \}$ and $X$. Remarkably, if $A$ contains a non-trivial two-sided norm-closed ideal $I$ then every bounded module operator $T \in {\rm End}_A(X)$ possesses the respective non-trivial ${\rm K}_I(X)$-$I$ imprimitivity subbimodule $IX$ of $X$ as an invariant subbimodule. This  adds a new perspective to the invariant subspace/submodule problem of bounded module operators $T$ on Hilbert C*-modules.
Moreover, there exist Noetherian and Artinian C*-algebras $A$ such that they contain infinite lattices of essential norm-closed two-sided ideals, cf.~\cite{HW_2012}. As another class of interesting examples one can mention C*-algebras $A$, the local multiplier algebra $M_{loc}(A)$ of which cannot be isometrically identified with $M_{loc}(M_{loc}(A))$, a phenomenon which continues sometimes even to higher order derivations of local multiplier algebras, cf.~\cite{M_2013}.

\begin{thm}
    Let $A$ and $B$ two (non-unital, in general) strongly Morita equivalent C*-algebras and $\{ X, \langle .,. \rangle \}$ be a Hilbert $B$-$A$ bimodule with faithful compact right action of $B$. Let $I \subseteq A$ and $J \subseteq B$ be essential norm-closed two-sided ideals. 
Then any bounded $A$-module operator and any bounded $B$-module operator maps any norm-closed Hilbert $J$-$I$ subbimodule of $X$ with compatible faithful compact right action of $J$ into itself. Therefore, the selected set of invariant Hilbert $J$-$I$ subbimodules with compatible faithful compact right action of $J$ cannot rule out any bounded one-sided module operator from the set of its preservers, and this set of all Hilbert $J$-$I$-subbimodules of $X$ with compatible faithful compact right action of $J$ is invariant under the action of any bounded left $A$-linear or right $B$-linear operator. 
\end{thm}

\begin{proof}
Since the situation is symmetric with respect to $A$ and $B$, we consider $B$ as the set of ``compact'' $A$-linear operators on $X$. For the Banach algebra of all right multipliers of $B$ we have the Banach algebra isomorphism $RM(B)={\rm End}_A(X)$ by \cite[Thm.~1.5]{Lin_1991}. So, for any element $T \in {\rm End}_A(X)$ there exists a net $\{ b_\alpha : \alpha \in K \} \subset B$ such that $\{ bb_\alpha : \alpha \in K \}$ converges to $bT$ for any $b \in B$ with respect to the $B$-norm. 

Fix an $J$-$I$-submodule $Y$ of $X$ with the C*-valued inner product induced from $X$. The faithful compact right action of $B$ via $\phi: B \to {\rm K}_A(X)$ reduces to a compatible faithful compact right action of $J$ on $X$, just as for $Y$.
Let $\{ u_\beta : \beta \in L \}$ be an approximate identity of $J$. For any $y \in Y$ the norm-limit of the net $\{y \phi(u_\beta) : \beta \in L \} \subset Y$ exists inside $Y$ and equals to $y$. For a fixed element  $y\phi(u_\beta)$ the net $\{ y\phi(u_\beta b_\alpha) :\alpha \in L \}$ converges to $y (\phi(u_\beta) T) \in Y$ with respect to the $Y$-norm by the right $J$-invariance of $Y$. Finally, the net $\{ (y \phi(u_\beta)) T : \beta \in L \} \subset Y$ converges to $y T \in Y$ with respect to the $Y$-norm by the norm-continuity of $T$. This demonstrates the invariance of $Y$ under the right action of $RM(B)={\rm End}_A(X)$. In other words, the right $B$-invariance of $Y$ forces the right $RM(B)$-invariance of $Y$. 
\end{proof}

Note that the proof would work equally well for non-norm-closed subbimodules $Y$ as long as those are full as a left Hilbert $A$-module and as a right Hilbert $B$-module via $\phi:B \to {\rm K}_A(X)$ at the same time. The centers of $B={\rm End}_A(X)$ and of $A={\rm End}_B(X)$ are $*$-isomorphic to the centers $Z(M(A))=Z(M(B))$ of their (one-sided) multiplier algebras with respect to standard identifications of module operators on $X$. 

Now, we investigate bounded bijective one-sided module operators that preserve both the left/right C*-valued inner product values on a $B$-$A$ imprimitivity bimodule. 

\begin{thm}
Let $A$ and $B$ two (non-unital, in general) strongly Morita equivalent C*-algebras and $\{ X, \langle .,. \rangle \}$ be an $B$-$A$ imprimitivity bimodule. Then the set of all bijective bounded $A$-linear operators on $X$ which preserve both the $A$-valued and $B$-valued inner product values coincides with the subset $\{ u \cdot {\rm id}_X : u \in Z(M(A))=Z(M(B)) \,\, {\rm unitary} \}$. By symmetry, the same result is true if one takes the $B$-valued inner product on $X$ as primary and looks for respective bijective bounded $B$-linear operators on $X$.
\end{thm}

\begin{proof}
Since the bijective bounded $A$-linear operator $T: X \to X$ preserves the $A$-valued inner product $\langle .,. \rangle_A$ it has to be a surjective isometry by \cite[Thm.]{Lance_1994}. So it is adjointable with adjoint $T^*=T^{-1}$. Switching to the second ${\rm K}_A(X)$-valued inner product derived from the initial $A$-valued inner product we obtain the requested equality as $\theta_{x,y} = \theta_{T(x),T(y)} = T\theta_{x,y}T^* = T\theta_{x,y}T^{-1}$ for any $x,y \in X$. Since the linear hull of the operators $\{\theta_{x,y} : x,y \in X \}$ is norm-dense in ${\rm K}_A(X)$ and $T$ induces a $*$-automorphism of ${\rm K}_A(X)$ as only shown we get the equality $T\theta_{x,y}=\theta_{x,y}T$ for any $x,y \in X$. So, $T= u \cdot {\rm id}_X$ with $u \in Z(M(A))$ such that $uu^*=1_A$, and by the bijectivity of $T$, or by the commutativity of $Z(M(A))$, also $u^*u=1_A$. So $u$ is a unitary element. 

Conversely, $\langle ux,uy \rangle_A = u \langle x,y \rangle_A u^* = \langle x,y \rangle_A uu^* = \langle x,y \rangle_A$  for any $x,y \in X$ and any unitary element of $Z(M(A))$. Also, $\theta_{ux,uy} = u\theta_{x,y}u^* = \theta_{x,y}uu^*= \theta_{x,y}$ for any $x,y \in X$ and any unitary element of $Z(M(A)) = Z({\rm End}_A(X))$ identifying $u$ with $u \cdot {\rm id}_X$.
\end{proof}

Note that the equality $\langle x,y \rangle_A z = x \langle y,z \rangle_B$ is preserved for any bijective adjointable bounded $A$-linear operator $T$ on any $B$-$A$ imprimitivity bimodule $\{ X, \langle .,. \rangle \}$ since the arbitrarily selected elements $x,y,z \in X$ simply can be replaced by elements $T(x),T(y),T(z)$. We should require adjointability of $T$ to stay in the unitary equivalence class of $X$, and consequently, in the $*$-isomorphism class of the respective resulting C*-algebra of ``compact'' operators. For bijective non-adjointable $T$ on $X$ we would compare two unitarily non-isomorphic C*-valued inner product structures on $X$, and therefore, two different equalities. Both these equalities characterize $X$ as an imprimitivity bimodule of $A$ with the respective non-$*$-isomorphic C*-algebras ${\rm K}_A(X)$.

\medskip
To get more significant preservation results one has to specify other intermediate sets of preferably one-sided invariant submodules of Hilbert C*-modules to be investigated. On the other side, the theory of (modular) eigenvalues and eigen-submodules of bounded module operators on certain classes of Hilbert C*-modules becomes much more interesting and diverse. To investigate more general left-Hilbert $B$-$A$ correspondences and their preservers as bimodules for any two C*-algebras $A$ and $B$ one needs some kind of $*$-representation theory of C*-algebras $B$ in strongly Morita equivalent to $A$ C*-algebras and their respective multiplier algebras. 

\smallskip \noindent
{\bf Acknowledgement:} 
The core ideas of this work were obtained during the UK Operator Algebras 
Conference 2025 in Belfast and the Satellite Workshop
"Crossed products and groupoid C*-algebras" there. 
The author would like to thank the Isaac Newton Institute 
for Mathematical Sciences, Cambridge, for support and 
hospitality during the programme "Topological groupoids 
and their C*-algebras" where work on this paper was 
undertaken. This work was supported by EPSRC grant 
no.~EP/Z000580/1.

\end{document}